\documentclass[11pt,reqno]{amsart}
\usepackage{latexsym,amsmath}
\usepackage{amssymb}
\usepackage{setspace}
\usepackage[left=3cm,top=3cm,right=3cm,bottom=3cm]{geometry}
\usepackage{amsthm}
\usepackage{epsfig}
\usepackage{graphicx}
\usepackage[usenames,dvipsnames]{color}
\usepackage{mathrsfs}

\begin{document}
\newtheorem{theorem}{Theorem}[section]
\newtheorem{lemma}[theorem]{Lemma}
\newtheorem{definition}[theorem]{Definition}
\newtheorem{conjecture}[theorem]{Conjecture}
\newtheorem{proposition}[theorem]{Proposition}
\newtheorem{algorithm}[theorem]{Algorithm}
\newtheorem{corollary}[theorem]{Corollary}
\newtheorem{observation}[theorem]{Observation}
\newtheorem{claim}[theorem]{Claim}
\newtheorem{problem}[theorem]{Open Problem}
\newtheorem{remark}[theorem]{Remark}
\newcommand{\noin}{\noindent}
\newcommand{\ind}{\indent}
\newcommand{\om}{\omega}
\newcommand{\I}{\mathcal I}
\newcommand{\ppp}{\mathfrak P}
\newcommand{\N}{{\mathbb N}}
\newcommand{\LL}{\mathbb{L}}
\newcommand{\R}{{\mathbb R}}
\newcommand{\E}[1]{\mathbb{E}\left[#1 \right]}
\newcommand{\V}{\mathbb Var}
\newcommand{\Prob}{\mathbb{P}}
\newcommand{\eps}{\varepsilon}

\newcommand{\Tv}{P}

\newcommand{\mT}{\mathcal{T}}
\newcommand{\mS}{\mathcal{S}}
\newcommand{\mA}{\mathcal{A}}
\newcommand{\mB}{\mathcal{B}}

\newcommand{\Cyc}[1]{\mathrm{Cyc}\left(#1\right)}
\newcommand{\Seq}[1]{\mathrm{Seq}\left(#1\right)}
\newcommand{\Mul}[1]{\mathrm{Mul}\left(#1\right)}
\newcommand{\Mulo}[1]{\mathrm{Mul}_{>0}\left(#1\right)}
\newcommand{\Set}[1]{\mathrm{Set}\left(#1\right)}
\newcommand{\Setd}[1]{\mathrm{Set}_{d}\left(#1\right)}

\newcommand{\dg}{d_{G}}
\newcommand{\de}{d_{E}}

\newcommand{\real}{\ensuremath {\mathbb R} }
\newcommand{\ent}{\ensuremath {\mathbb Z} }
\newcommand{\nat}{\ensuremath {\mathbb N} }

\newcommand{\RG} {\ensuremath{\mathscr G(n,r)}}
\newcommand{\RGuv} {\ensuremath{\widetilde{\mathscr G}(n,r)}}
\newcommand{\RGR} {\ensuremath{\widetilde{\mathscr G}_{R,u,v}(n,r)}}
\newcommand{\SR} {\ensuremath{\mathcal S}}

\title{Burning graphs---a probabilistic perspective}

\author{Dieter Mitsche}
\address{Universit\'{e} de Nice Sophia-Antipolis, Laboratoire J-A Dieudonn\'{e}, Parc Valrose, 06108 Nice cedex 02}
\email{\texttt{dmitsche@unice.fr}}

\author{Pawe\l{} Pra\l{}at}
\address{Department of Mathematics, Ryerson University, Toronto, ON, Canada}
\email{\tt pralat@ryerson.ca}

\author{Elham Roshanbin}
\address{Department of Mathematics and Statistics, Dalhousie University, Halifax, NS, Canada}
\email{\tt e.roshanbin@dal.ca}

\thanks{The first author acknowledges support of the PROCOPE-DAAD project RanConGraph (ref. 57134837). The second author acknowledges support from NSERC and Ryerson University.}


\begin{abstract}
In this paper, we study a graph parameter that was recently introduced, the burning number, focusing on a few probabilistic aspects of the problem. The original burning number is revisited and analyzed for binomial random graphs $\mathcal{G}(n,p)$, random geometric graphs, and the Cartesian product of paths. Moreover, new variants of the burning number are introduced in which a burning sequence of vertices is selected according to some probabilistic rules. We analyze these new graph parameters for paths.
\end{abstract}

\maketitle

\section{Introduction and results}

In this paper, we study a graph parameter, the burning number, that was recently introduced as a simple model of spreading social influence~\cite{BJR}. Our focus is on a few probabilistic aspects of this problem. Randomness is coming from two possible sources: the original burning number is investigated for random graphs, and some new variants are introduced in which some probabilistic rules are introduced, replacing deterministic ones that are embedded in the original model. The problem is inspired by many well-known processes on graphs and their probabilistic counterparts, including \emph{firefighter}~\cite{fire1, fire2} (see the survey~\cite{fire3} for an overview of many deterministic results), \emph{cleaning process}~\cite{brush1, brush2, brush3}, \emph{bootstrap percolation} (see the survey~\cite{bootstrap}), and synchronous \emph{rumour-spreading} (see for example~\cite{abbas} and the references therein).

\bigskip

Let us start with the original process introduced recently in~\cite{BJR}. Given a finite, simple, undirected graph $G$, the burning process on $G$ is a discrete-time process defined as follows. Initially, at time $t=0$ all vertices are unburned. At each time step $t \geq 1$, an unburned vertex is chosen to burn (if such a vertex is still available); if a vertex is burned, then it remains in that state until the end of the process. Once a vertex is burned in round $t$, in round $t+1$ each of its unburned neighbours becomes burned. The process ends when all vertices of $G$ are burned. The \emph{burning number} of a graph $G$, denoted by $b(G)$, is the minimum number of rounds needed for the process to end. It is obvious that for every connected graph $G$ we have that $b(G) \le D(G) + 1$, where $D(G)$ is the diameter of $G$.

\subsection{Binomial random graphs} Our first result is for random graphs. The \emph{binomial random graph} $\mathcal{G}(n,p)$ is defined as a random graph with vertex set $[n]=\{1,2,\dots, n\}$ in which a pair of vertices appears as an edge with probability $p$, independently for each pair of vertices. As typical in random graph theory, we shall consider only asymptotic properties of $\mathcal{G}(n,p)$ as $n\rightarrow \infty$, where $p=p(n)$ may and usually does depend on $n$. See~\cite{JLR} for more details on the model.

Throughout the paper, we use the following standard notation for the asymptotic behaviour of sequences of non-negative numbers $a_n$ and $b_n$: $a_n=O(b_n)$ if $\limsup_{n\to\infty}a_n/b_n\leq C< \infty$;  $a_n=\Omega(b_n)$ if $b_n=O(a_n)$; $a_n=\Theta(b_n)$ if $a_n=O(b_n)$ and $a_n=\Omega(b_n)$; $a_n=o(b_n)$ if $\lim_{n\to\infty} a_n/b_n = 0$, and $a_n=\omega(b_n)$ if $b_n=o(a_n)$. We also use the notation $a_n \ll b_n$ for $a_n=o(b_n)$ and $a_n \gg b_n$ for $b_n=o(a_n)$. Finally,  a sequence of events $H_n$ holds \emph{asymptotically almost surely} (\emph{a.a.s.}) if $\lim_{n\to\infty}\Pr(H_n)=1$. All logarithms in this paper are natural logarithms.

\bigskip

Now, we are ready to state the main result for binomial random graphs.

\medskip

\begin{theorem}~\label{thm:mainRandom}
Let $G \in \mathcal{G}(n,p)$, $\eps >0$, and $\omega = \omega(n) \to \infty$ as $n \to \infty$ but $\omega = o(\log \log n)$. 

Suppose first that 
$$
d=d(n)=p(n-1) \gg \log n \textrm{ \ \ \ and \ \ \  } p \le 1-(\log n + \log \log n + \omega)/n.
$$  
Let $i \ge 2$ be the smallest integer such that
$$
d^i/n - 2 \log n \to \infty.
$$
Then, the following property holds a.a.s. 
\begin{equation}\label{eq:i}
b(G) = 
\begin{cases}
i & \text{ if } \ \ d^{i-1}/n \ge (1+\eps) \log n\\
i \text{ or } i+1 & \text{ if } \ \ (1-\eps) \log d \le d^{i-1}/n < (1+\eps) \log n\\
i+1 & \text{ if } \ \ d^{i-1}/n < (1-\eps) \log d.
\end{cases}
\end{equation}

If
$$
1-(\log n + \log \log n + \omega)/n < p \le 1-(\log n + \log \log n - \omega)/n,
$$  
then a.a.s. 
\begin{equation}\label{eq:ii}
b(G) = 2 \text{ or } 3.
\end{equation}

Finally, if
$$
p > 1-(\log n + \log \log n - \omega)/n,
$$  
then a.a.s.
\begin{equation}\label{eq:iii}
b(G) = 2.
\end{equation}
\end{theorem}

\subsection{Random geometric graphs}  Our next result deals with random geometric graphs. Given a positive integer $n$, and a non-negative real $r$, we consider a \emph{random geometric graph} $G\in\RG$ defined as follows. The vertex set $V$ of $G$ is obtained by choosing $n$ points independently and uniformly at random in the square $\SR = [0,1]^2$. Note that, with probability $1$, no point in $\SR$ is chosen more than once, and thus we may assume $|V|=n$. For notational purposes, we identify each vertex $v \in V$ with its corresponding geometric position $v=(v_x,v_y)\in\SR$, where $v_x$ and $v_y$ denote the usual $x$- and $y$-coordinates in $\SR$. Finally, the edge set of $G\in\RG$ is constructed by connecting each pair of vertices $u$ and $v$ by an edge if and only if $d_E(u,v)\le r$, where $d_E$ denotes the Euclidean distance in $\SR$. We also use the graph distance $d_G(u,v)$ to denote the number of edges on a shortest path between $u$ and $v$. As in the case of binomial random graphs, we shall consider only asymptotic properties of $\RG$ as $n\rightarrow \infty$, where $r=r(n)$ may and usually does depend on $n$. See~\cite{Penrose} for more details on the model.

\medskip

It is well known that  $r_c=\sqrt{\log n/(\pi n)}$ is a sharp threshold function for the connectivity of a random geometric graph (see~e.g.~\cite{Penrose97,Goel05}). This means that for every $\varepsilon>0$, if $r\le(1-\varepsilon)r_c$, then $\RG$ is a.a.s.\ disconnected, whilst if $r\ge(1+\varepsilon)r_c$, then it is a.a.s.\ connected. We have the following theorem:

\begin{theorem}\label{thm:mainRGG}
Let $G\in\RG$ with $r \geq Cr_c$ for $C$ being a sufficiently large constant. Then, a.a.s.,
$$b(G)=\Theta\left(r^{-2/3}\right).$$
\end{theorem}

\bigskip

\subsection{Grids} We also deal with the Cartesian product of two paths, which is the only deterministic result shown in this paper. In~\cite{BJR}, the burning number of the Cartesian product of two paths was studied, including non-symmetric cases, that is, the product of two paths of different lengths. However, only the order of this graph parameter was found. Here, we investigate the asymptotic behaviour of the burning number for grids.
\begin{theorem}\label{thm:mainGrid}
$$b(P_m \square P_n) =
\begin{cases}
 (1+o(1)) (3/2)^{1/3} (mn)^{1/3}, & \text{ if } \sqrt{n} \ll m \le n, \\
 \Theta(\sqrt{n}), & \text{ if } m = O(\sqrt{n}).
\end{cases}
$$
\end{theorem}

\subsection{Cost of drunkenness}
Finally, we also investigate the following variant of the problem, inspired by a similar variant of the game of cops and robbers~\cite{drunk2, drunk1}. For a given graph $G=(V,E)$, instead of selecting the sequence of vertices $(x_1, x_2, \ldots, x_k)$ so that the number of burning vertices at the end of the process is maximized, the sequence is selected randomly as it was generated by a drunk person. There are at least three natural notions of randomness (levels of ``drunkenness'') one can consider. 
\begin{itemize}
\item [(i)] At time $i$ of the process, $x_i$ is selected uniformly at random from $V$; that is, for each $v \in V$, $\Prob(x_i = v) = 1/n$. (In particular, it might happen that $v$ is already burning or maybe even was selected earlier in the process.)
\item [(ii)] At time $i$ of the process, $x_i$ is selected uniformly at random from those vertices that were not selected before; that is, for each $v \in V$ that was not selected earlier in the process, $\Prob(x_i = v) = 1/(n-i+1)$. (However, it still might happen that $v$ is already burning.)
\item [(iii)] At time $i$ of the process, $x_i$ is selected uniformly at random from those vertices that are not burning at time $i$.
\end{itemize}
Let $b_1(G)$, $b_2(G)$, and $b_3(G)$ be the random variables associated with the first time all vertices of $G$ are burning, for the three variants of selecting vertices mentioned above. Clearly, each of these random variables are at least $b(G)$, and the variables can be easily coupled to see that
\begin{equation}\label{eq:3b}
b_1(G) \ge b_2(G) \ge b_3(G) \ge b(G).
\end{equation}
For $j \in [3]$, we define $c_j(G) = b_j(G) / b(G) \ge 1$ to be the \emph{cost of drunkenness} of $G$ for the three processes we focus on. 

\medskip

We illustrate the cost of drunkenness on the path $P_n$ on $n$ vertices; suppose that $V(P_n) =\{v_1, v_2, \ldots, v_n\}$, with $v_i$ being adjacent to $v_{i-1}$ for $2\leq i\leq n$. It is easy to see that $b(P_n) = \lceil \sqrt{n} \rceil$ (see~\cite{BJR} for this and more results). The following result shows that the first two costs of drunkenness ($c_1(P_n)$ and $c_2(P_n)$) are asymptotically equal to each other, and that both $b_1(P_n)$ and $b_2(P_n)$ are much bigger than $b(P_n)$. On the other hand, the third cost of drunkenness is much smaller, and we have $b_3(P_n)=\Theta(b(P_n))$. More precisely, we have the following result:

\begin{theorem}\label{thm:Path1}
A.a.s.\ the following holds:
\begin{itemize}
\item [(i)] $b_1(P_n) = (1+o(1)) b_2(P_n) = (1+o(1)) \sqrt{n \log n / 2},$ and thus \\
$c_1(P_n) = (1+o(1)) c_2(P_n) = (1+o(1)) \sqrt{\log n / 2}.$
\item [(ii)]
$
b_3(P_n) = \Theta( \sqrt{n}),
$
and thus
$
c_3(P_n) = \Theta(1).
$
\end{itemize}
\end{theorem}

\subsection{Organization of the paper}

All our main results, Theorem~\ref{thm:mainRandom}, Theorem~\ref{thm:mainRGG}, Theorem~\ref{thm:mainGrid}, and Theorem~\ref{thm:Path1}, are proved independently in Section~\ref{sec:Gnp}, Section~\ref{sec:geometric}, Section~\ref{sec:grid}, and Section~\ref{sec:drunkenness}. 

\section{Proof of Theorem~\ref{thm:mainRandom}}\label{sec:Gnp}

In this section we consider $G \in \mathcal{G}(n,p)$. In order to bound the burning number of $G$ from above, we will make use of the following result for random graphs, see~\cite[Corollary~10.12]{bol}.

\begin{lemma}[\cite{bol}, Corollary 10.12]\label{lem:diameter}
Suppose that $d = p(n-1) \gg \log n$, $1-p \gg n^{-2}$, and 
\[
d^i/n - 2 \log n \to \infty \text{ \ \ \ \ and \ \ \ \ } d^{i-1}/n - 2 \log n \to -\infty.
\]
Then the diameter of  $G \in \mathcal{G}(n,p)$ is equal to $i$ a.a.s.
\end{lemma}

From the proof of this result, we have the following corollary.

\begin{corollary}\label{cor:diameter}
Suppose that $d = p(n-1) \gg \log n$ and that
 $$d^i/n - 2 \log n \to \infty.$$ Then the diameter of $G \in \mathcal{G}(n,p)$ is at most $i$ a.a.s. 
\end{corollary}
In order to obtain lower bounds and the upper bound in the first case of (\ref{eq:i}), we will need the following expansion lemma investigating the shape of typical neighbourhoods of vertices. Before we state the lemma we need a few definitions. For any $j \geq 0$, let us denote by $N(v,j)$ the set of vertices at distance at most $j$ from $v$, and by $S(v,j)$ the set of vertices at distance exactly $j$ from $v$ (note that $S(v,0)=\{v\})$.

\begin{lemma}\label{lem:expansion}
Let $G = (V,E) \in \mathcal{G}(n,p)$ and let $\eps > 0$. Suppose that $d=p(n-1)$ is such that  $\log n \ll d = o(n),$ and let $i \ge 1$ be the largest integer such that $d^{i} = o(n).$ Then, a.a.s.\ the following properties hold:
\begin{itemize}
\item[(i)] for all $j=1,2,\ldots,i$ and all $v \in V$,
$$
|N(v,j)| = |S(v,j)| (1+O(1/d)) = d^j(1+o(1)),
$$ 
\item[(ii)] there exists $v \in V$ such that $N(v,i+1) = V$, provided that $d^{i+1}/n \ge (1+\eps) \log n$,
\item[(iii)] for all $v \in V$, 
$$|V \setminus N(v,i+1)| = e^{-c (1+o(1))} n,$$
provided that $c = c(n) = d^{i+1}/n \le (1-\eps) \log n$.
\end{itemize}
\end{lemma}
\begin{proof}
Let $v \in V$ and consider the random variable $X = X(v) = |S(v,1)|$. It is clear that $X \in \textrm{Bin}(n-1,p)$ so we get that $\E{X} = d.$ A consequence of Chernoff's bound (see e.g.~\cite[Corollary~2.3]{JLR}) is that 
\begin{equation}\label{eq:chern}
\Prob \Big( |X-\E X| \ge \eps \E X \Big) \le 2\exp \left( - \frac {\eps^2 \E X}{3} \right)  
\end{equation}
for  $0 < \eps < 3/2$. Hence, after taking $\eps = 2 \sqrt{\log n / d}$, we get that with probability $1-o(n^{-1})$ we have 
$$
X = \E{X} (1+O(\eps)) = d (1+o(1)).
$$
 
We will continue expanding neighbourhoods of $v$ using the BFS (breath-first search) procedure. Suppose that $S=N(v,j)$ for some $j \ge 1$, $s=|S|$, and our goal is to estimate the size of $N(v,j+1)$. Consider the random variable $X = X(S)$ counting the number of vertices outside of $S$ with at least one neighbour in $S$.  (Note that $X$ is only a lower bound for $|N(v,j+1)|$, since, in fact, $X=|N(v,j+1) \setminus N(v,j)|$. We will show below that $|N(v,j+1)| \leq (1+O(1/d))X$.) We will bound $X$ in a stochastic sense. There are two things that need to be estimated: the expected value of $X$, and the concentration of $X$ around its expectation. 

It is clear that  
\begin{eqnarray*}
\E X &=& \left( 1 - \left(1- \frac {d}{n-1} \right)^s \right) (n-s) \\
&=& \left( 1 - \exp \left( - \frac {ds}{n} (1+O(d/n)) \right) \right) (n-s) \\
&=& \frac {ds}{n} (1+O(ds/n)) (n-s) 
= ds (1+O(ds/n)). 
\end{eqnarray*}
It follows that $\E X = ds (1+O(\log^{-1} n))$ provided $ds \le n/ \log n$, and $\E X = ds (1+o(1))$ for $ds = o(n)$.  
We next use Chernoff's bound (\ref{eq:chern}) which implies that with probability $1-o(n^{-2})$ we have $\big| X - d s \big| \le \eps d s$ for $\eps = 3 \sqrt{\log n/(ds)}$. In particular we get that with probability $1-o(n^{-2})$ we have $X = ds(1+O(\log^{-1} n))$, provided $\log^3 n < ds < n/\log n$, and $X = ds(1+o(1))$ for $ds = o(n)$. We consider the BFS procedure up to the $i$'th neighbourhood provided that $d^i =o(n)$. Note that this implies that $i = O(\log n /\log \log n)$. Then the cumulative multiplicative error term  is 
$$ (1+o(1))^3 (1+O(\log^{-1}n))^i = (1+o(1)) (1+O(i \log^{-1}n)) = (1+o(1)).$$ 
(Note that it might take up to 2 iterations to reach at least $\log^3 n$ vertices to be able to use the error of $(1+O(\log^{-1}n))$ and possibly one more iteration when the number of vertices reached is $o(n)$ but is larger than $n/\log n$.) In other words, by a union bound over $1 \le j \le i$, with probability $1-o(n^{-1})$, we have $|N(v,j) \setminus N(v,j-1)|=d^j (1+o(1))$ for all $1\le j \le i$, provided that $d^i = o(n)$. By taking a union bound one more time (this time over all vertices), a.a.s.\ the same bound holds for all $v \in V$. Therefore, a.a.s., for all $v \in V$ and all $1 \leq j \leq i$, $|N(v,j)|=|S(v,j)|(1+O(1/d))$ and also $|N(v,j)|=d^j(1+o(1))$, and Part (i) follows.

We continue investigating the neighbourhood of a given vertex $v \in V$. It follows from the previous part that $s = |N(v,i)| = d^i (1+o(1))$ with probability $1-o(n^{-1})$. Let $c = c(n) = d^{i+1}/n = \Omega(1)$. This time, it is easer to focus on the random variable $Y$ counting the number of vertices outside of $N(v,i)$ that have \emph{no} neighbour in $N(v,i)$. It follows that 
\begin{eqnarray*}
\E Y &=& \left(1- \frac {d}{n-1} \right)^s (n-s) \\
&=& \exp \left( - \frac {ds}{n} (1+O(d/n)) \right) n (1+o(1)) 
= e^{-c (1+o(1))} n.
\end{eqnarray*}
If $c \ge (1+\eps) \log n$ for some $\eps>0$, then $\E Y \le n^{-\eps + o(1)} = o(1)$ and so a.a.s.\ $Y = 0$ by Markov's inequality, and (ii) holds. On the other hand, if $c \le (1-\eps) \log n$ for some $\eps>0$, then $\E Y \ge n^{\eps+o(1)}$ and so, using (\ref{eq:chern}) one more time, we get that with probability $1-o(n^{-1})$, $Y = (1+o(1)) e^{-c (1+o(1))} n$. Part (iii) holds by taking a union bound over all $v \in V$. 
\end{proof}

\bigskip

Now, the proof of Theorem~\ref{thm:mainRandom} is an easy consequence of the previous lemma together with some well-known results.
Let us first deal with the case $d=o(n)$. Before we start considering the three cases of (\ref{eq:i}), let us notice that it follows from the definition of $i$ that $d^{i-2}/n = O(\log n / d) = o(1)$. Since we aim for a result that holds a.a.s., we may assume that $G$ satisfies (deterministically) the properties from Lemma~\ref{lem:expansion} and Corollary~\ref{cor:diameter}.

Suppose that $d^{i-1}/n \ge (1+\eps) \log n$. It follows from Lemma~\ref{lem:expansion}(ii) that there exists $v \in V$ such that $N(v,i-1) = V$. In order to show $b(G) \le i$, it suffices to start burning the graph from $v$. On the other hand, by Lemma~\ref{lem:expansion}(i), regardless of the strategy used, the number of vertices burning after $i-1$ steps is at most 
$$
\sum_{j=0}^{i-2} d^j (1+o(1)) = d^{i-2} (1+o(1)) = o(n).
$$
Hence,  $b(G) \ge i$, and the first case is done for $d=o(n)$.

Suppose now that $(1-\eps) \log d \le d^{i-1}/n < (1+\eps) \log n$. Exactly the same argument as in the first case gives $b(G) \geq i$;  the number of vertices burning after $i-1$ steps is at most $d^{i-2}(1+o(1)) = o(n)$. By Corollary~\ref{cor:diameter}, the diameter of $G$ is at most $i$, and so $b(G) \leq i+1$. The second case is done for $d=o(n)$.

Finally, suppose that $c = d^{i-1}/n < (1-\eps) \log d$. It follows from Lemma~\ref{lem:expansion}(i) and (iii) that, no matter which burning sequence is used, after $i$ steps, the number of vertices \emph{not} burning is at least
\begin{eqnarray*}
e^{-c(1+o(1))} n - \sum_{j=0}^{i-2} d^j (1+o(1)) &\ge& \exp \Big(-(1-\eps + o(1)) \log d \Big) n - d^{i-2} (1+o(1)) \\
&=& \frac {n d^{\eps+o(1)}}{d} - \frac {cn}{d} (1+o(1)) \\
&\ge& \frac {n d^{\eps+o(1)}}{d} - \frac {n \log d}{d} \ge \frac {n d^{\eps+o(1)}}{2d} \to \infty.
\end{eqnarray*}
Hence, $b(G) \ge i+1$ and, since the diameter is at most $i$, we in fact have $b(G) = i+1$ and the third case is finished for $d=o(n)$.

\medskip

Now, let us consider $p = \Omega(1)$ and $p \le 1-(\log n + \log \log n + \omega)/n$. For this range of the parameter $p$, the diameter of $G$ is a.a.s.\ equal to 2, $i=2$, and we are still in the third case of~(1). Since $b(G)=1$ if only if the graph consists of one vertex, it is clear that for this range of $p$ the only two possibilities for the burning number are 2 and 3. Note that $b(G)=2$ if and only if there exists $v\in V$ such that $N(v,1)$ covers all but perhaps one vertex. Indeed, if we start burning the graph from $v$, all but at most one vertex is burning in the next round and the remaining vertex (in case it exists) can be selected as the next one to burn. On the other hand, if no such $v \in V$ exists, then no matter which vertex is selected as the starting one, there is at least one vertex not burning in the next round. This sufficient and necessary condition for having $b(G)=2$ is equivalent to the property that the complement of $G$ has a vertex of degree at most one (that is, the minimum degree of the complement of $G$ is at most one). It is well-known that the threshold for having minimum degree at least 2 is equal to $p_0 = (\log n + \log \log n)/n$. Hence, if $1-p \ge (\log n + \log \log n + \omega)/n$, then a.a.s.\ the minimum degree in the complement of $G$ is at least two, and so $b(G) = 3$ a.a.s. This finishes the proof of~\eqref{eq:i}.

For the range of $p$ given in~\eqref{eq:ii}, note that a.a.s.\ the minimum degree of the complement of $G$ is equal to one or two (recall that $\omega = o(\log \log n)$, and so the complement of $G$ is a.a.s.\ connected). Hence $b(G) \in \{2,3\}$ a.a.s. Finally, the range for $p$ given in~\eqref{eq:iii}) is below the critical window for having minimum degree $2$ in the complement of $G$; it follows that a.a.s.\ the minimum degree of the complement of $G$ is at most one, and so $b(G) = 2$ a.a.s. The proof of Theorem~\ref{thm:mainRandom} is finished.

\newpage

We get immediately the following corollary: 
\begin{corollary}
Let $d=d(n)=p(n-1) \gg \log n$ and let $G \in \mathcal{G}(n,p)$. Then, a.a.s.
$$
b(G) - D(G) \in \{0, 1\}.
$$
\end{corollary}

Let us mention that, in some sense, the corollary is best possible. If, for example, $i \in \N$ is such that $d^{i-1}/n = \frac 32 \log n$, then a.a.s.\ $b(G)=D(G)=i$. On the other hand, if, say, $d^{i-1}/n = 1$, then a.a.s.\ $b(G)$ is larger than $D(G)$ (in fact, a.a.s.\ $b(G) = i+1$ and $D(G) = i$).

\section{Proof of Theorem~\ref{thm:mainRGG}}\label{sec:geometric}

This section deals with our result on random geometric graphs. 
We show lower and upper bounds separately and start with the lower bound.

\begin{lemma}\label{lem:lb}
Let $G\in\RG$ with $r \geq r_c$. Then, a.a.s.,
$$b(G)=\Omega\left(r^{-2/3}\right).$$
\end{lemma}
\begin{proof}
Suppose that we are given a graph $G \in \RG$. Tessellate $\SR$ into sub-squares of side length $s=4 \sqrt{\log n/n} = \Theta(r_c)$ which we call \emph{cells} (the rightmost column and the bottommost row might contain slightly bigger cells, in case en equal tessellation is not possible). For a cell $C$, denote by $X_C$ the random variable counting the number of vertices of $G$ inside $C$. Clearly, $X_C \sim {\rm Bin}(n,p)$ with $p$ being the area of $C$, and thus, by~\eqref{eq:chern},
\begin{equation*}
\Prob(X_C\le \E {X_C}/2) \le 2e^{-(1/2)^2 s^2 n/3} = o(n^{-1}).
\end{equation*}
By a union bound over all $(\lfloor 1/s \rfloor)^2 =O(n/\log n)$ cells, a.a.s.\ this holds for all cells.
Suppose now that $b(G)=t \geq 1$ for some $t$ to be defined later. For $i=1,2,\ldots,t$, let $v_i$ to be vertex chosen to be burned at at the $i$-th step. Note that in order for a vertex $w$  to be burned by time $t$ due to the initial burning of $v_i$, we must have $d_E(w,v_i) \leq (t-i)r$. Letting $B_i$ denote the subarea of $\SR$ containing vertices that potentially can be burned due to the burning of $v_i$, we clearly have $\mbox{area}(B_i) \leq  (t-i)^2 r^2 \pi$. Hence, for $r=\omega(r_c)$ or $t-i=\omega(1)$ the number of cells that have non-empty intersection with $B_i$ is at most 
$$
(1+o(1))\frac{(t-i)^2 r^2 \pi}{s^2},
$$
and for $r\geq r_c$ and $t-i=\Theta(1)$ the number of cells with non-empty intersection with $B_i$ is at most
$$
C_0 \frac{(t-i)^2 r^2 \pi}{s^2},
$$
where $C_0 > 1$ is a large enough absolute constant. (For example, $C_0=25 \cdot 16$ is sufficient. Indeed, one can distinguish between \emph{regular} cells that either contain $v_i$, the center of $B_i$, or are completely contained in $B_i$; the remaining cells having non-empty intersection with $B_i$ are called \emph{boundary} cells. Note that every boundary cell is at $L_1$-cell-distance at most $2$ of a regular cell of $B_i$, and thus the total number of cells is at most $25$ times the number of regular cells. Finally, note that the area of regular cells is at most $s^2/(\pi r_c^2) \cdot \mbox{area}(B_i) = 16 \cdot \mbox{area}(B_i)$; the extreme case corresponds to $r=r_c$ and $i=t-1$.) Thus, in any case, the number of cells having non-empty intersection with at least one region $B_i$ by time $t$ is at most
$$
\sum_{i=1}^t C_0 \frac{(t-i)^2 r^2 \pi}{s^2} \le C_0 \frac{t^3 r^2 \pi}{3s^2}.
$$
Therefore, since there are $(1+o(1)) s^{-2}$ cells in total, if (for example) $t \leq \left(2/ (C_0 \pi r^2) \right)^{1/3} $, at least one cell has empty intersection with $\bigcup_{i=1}^t B_i$ by time $t$. Since a.a.s.\ each cell contains at least one vertex, a.a.s.\ at least one vertex is not burning by time $t$ in $\RG$. \end{proof}

\medskip

For the upper bound, we will use the result from~\cite{Diaz}. The results there are stated in the model of a square of side length $\sqrt{n}$, but they can be easily translated to our setting. (In fact, in~\cite{Diaz}, a more precise result was shown, but for our purpose the following version is enough.)

\begin{theorem}[\cite{Diaz}] \label{thm:ownresult}
Let $G\in\RG$. The following holds a.a.s. If $r \geq Cr_c$ with $C$ sufficiently large, then there exists $C'=C'(C)$ such that for every pair of vertices $u,v \in V(G)$ we have 
$$ d_G(u,v) \le  C'\frac{d_E(u,v)}{r}. $$
\end{theorem}

We are now ready to prove the upper bound.
\begin{lemma}\label{lem:ub}
Let $G\in\RG$ with $r \geq Cr_c$ for $C$ sufficiently large. Then, a.a.s.,
$$b(G)=O\left(r^{-2/3}\right).$$
\end{lemma}
\begin{proof}  
Let $C'=C'(C)$ be such that by Theorem~\ref{thm:ownresult}(ii), a.a.s.\ for every pair of vertices $u,v$ we have $d_G(u,v) \le C'\frac{d_E(u,v)}{r}.$ For convenience, define $A=(3C' \sqrt{2})^{-1/3}$. As in the proof of the previous lemma, for a given $G \in \RG$, we consider a partition of $\SR$ into cells of side length $A r^{1/3}$ (as before, the rightmost column and the bottommost row might contain slightly bigger cells, in case an equal tessellation is not possible; clearly, the side length of these cells is at most $2Ar^{1/3}$). Note that $\SR$ contains $s = (\lfloor 1/(Ar^{1/3}) \rfloor)^2=\Theta(r^{-2/3})$ cells. Moreover, since each cell has area $\Theta(r^{2/3})$, it contains in expectation $\Theta(nr^{2/3})=\Omega(n^{2/3} \log^{1/3} n)$ many vertices (recall that $r_c=\sqrt{\log n/(\pi n)}$). By~\eqref{eq:chern}, together with a union bound over all $\Theta(r^{-2/3})=O((n/\log n)^{1/3})$ cells, a.a.s.\ each cell contains at least one vertex.

Now, during the first phase consisting of $s$ time steps, we choose in each cell one vertex to be burned. As we aim for an upper bound, we may assume that these vertices are the only ones burning at the end of the first phase. The second phase again consists of $s$ time steps. It suffices to show that each vertex that was selected during the first phase will make all vertices in its cell burned at the and of phase $2$. Note that for any two vertices $u,v$ in a cell, the Euclidean distance between them is at most $2\sqrt{2}Ar^{1/3}$. It follows that a.a.s., for any pair of vertices $u,v \in \RG$, 
$$
d_G(u,v) \leq C'\frac{d_E(u,v)}{r} \leq  2\sqrt{2}AC'r^{-2/3}  \leq s,
$$ 
where the last inequality follows from our choice of $A$. Hence, during the second phase, a.a.s.\ all vertices are burned, and the upper bound follows.
\end{proof}

\section{Proof of Theorem~\ref{thm:mainGrid}}\label{sec:grid}

The lower bound is straightforward. Note that for any vertex $v \in V(P_{m} \square P_n)$ the number of vertices of a ball of radius $r$ around $v$, $|N(v,r)|$, is equal to  
$$
|N(v,r)| \leq 1+4+8+\ldots+4r = 1+2(r+1)r.
$$
Hence, at time $k$ the number of vertices burning is at most the sum of the total number of vertices included in some ball of radii $0,1,\ldots, k-1$, which is is at most
\begin{eqnarray*}
f(k) &=& \sum_{r=0}^{k-1} ( 1+2r+2r^2 ) = 2k^3/3 + k/3.
\end{eqnarray*}
Clearly, if $f(k_0) < mn$, then $\lfloor k_0 \rfloor$ balls cannot cover the whole graph and so $b(P_{m} \square P_n) \ge \lceil k_0 \rceil$. The desired inequality holds for $k_0 = (3mn/2)^{1/3} - 1$, since we have
$$
f(k_0) = \frac 23 k_0^3 + \frac{k_0}{3} 
< \frac23 (k_0+1)^3
= mn.
$$
Hence, for $\sqrt{n} \leq m \leq n$ we obtain the following useful lower bound: 
$$b(P_m \square P_n) \ge (3/2)^{1/3} (mn)^{1/3} - 1.$$

On the other hand, it is easy to see that $b(P_m \square P_n) \ge b(P_n) \ge \sqrt{n}$. To see this, one can focus on a path on $n$ vertices on the border of the grid. As any ball of radius $r$ (centered on any vertex of the grid, not necessarily on the path we focus on) contains at most $2r+1$ vertices from the path, the total number of vertices on the path burning at time $\lfloor \sqrt{n} \rfloor$ is
$$
\sum_{r = 0}^{\lfloor \sqrt{n} \rfloor - 1} (2r+1) \le \frac {2 \sqrt{n}-1}{2} \sqrt{n} < n.
$$
It follows that $b(P_m \square P_n) \ge \sqrt{n}$, which is useful for $m < \sqrt{n}$.

\medskip

Now, let us move to the upper bound. Suppose first that $m = \gamma(n) \sqrt{n}$, where $1 \le \gamma = \gamma(n) \le \sqrt{n}$. We will show that it is possible to cover $P_{m} \square P_n$ with balls of radii between $k_1$ and $k_2$, where
$$
k_1 = 	\frac {(mn)^{1/3}}{\gamma^{1/6}} \ \ \ \ \  \text{ and } \ \ \ \ \  k_2 = \left( \frac 32 \right)^{1/3} (mn)^{1/3} \left( 1 + \frac {C}{\gamma^{1/6}} \right),
$$
where $C\ge 1$ is some large constant that will be determined soon. This clearly implies that $b(P_m \square P_n) \le k_2+1$. (Let us note that we are not optimizing the error term here, aiming for an easy argument.)

\begin{figure}[htbp]
\includegraphics[width=0.3\textwidth]{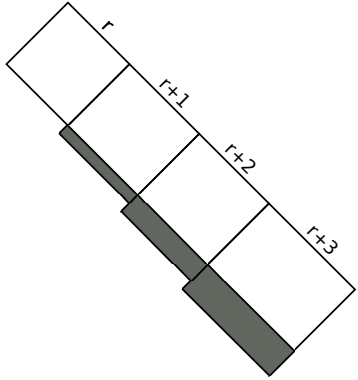} 
\caption{Covering the grid with diagonal strips \label{fig:strip} }
\end{figure}

We will cover the graph with diagonal `strips' using radii $k_1 \le r \le k_2$. More precisely, we choose the strips in the following way: the rightmost diagonal strip has the top right corner as its center, and is of radius $k_1$. Then, as long as the top border line of the grid (of length $n$) is not yet covered, do the following: given $i \geq 1$ strips already defined with $r_i$ being the last radius used in strip $i$, the $(i+1)$-st  strip is such that the topmost ball of the $(i+1)$-st strip has as its rightmost vertex the left neighbour of the leftmost vertex of the topmost ball of strip $i$, and has radius $r_i+1$. The radii of balls increase always by $1$ inside a strip (see Figure~\ref{fig:strip}) but the width of the $(i+1)$-st  strip is equal to $r_i+1$.  If the top border line of the grid is now completely covered but the left border line (of length $m$) is not yet completely covered, then we proceed as follows. Given $i \geq 1$ strips already defined with $r_i$ being the last radius used in strip $i$, choose the topmost vertex on the left border line that is not covered by the first ball of strip $i$ to be the topmost vertex of strip $i+1$; as before the strip has radius $r_i+1$, and the radii of balls always increase by $1$ inside a strip.

Let us concentrate on a strip consisting of balls of radii $r, r+1, r+2, \ldots$ (see Figure~\ref{fig:strip}). Instead of estimating the number of vertices covered (see white rectangles on the figure), it is easier to bound from above the number of vertices ``wasted'' (see grey rectangles on the figure), that is, vertices that will be (double) covered by the next strip or that will fall outside the boundary of the grid. Since the number of balls forming a strip whose smallest ball has radius $r$ is $O(m/r)$, the largest ball in the strip has radius $r+O(m/r) = (1+o(1)) r$ (as $r = \Omega(k_1) = \Omega(\gamma^{1/6} \sqrt{n})$ and $O(m/r) = O(m/k_1) = O(\gamma^{5/6}) = O(n^{5/12})$). Hence, the number of vertices wasted in the last ball of the strip is at most $(1+o(1)) r \cdot O(m/r) = O(m)$, and so the number of vertices wasted for this strip is 
$$
O(r^2) + O(m^2/r) = O(k_2^2) + O(m^2 / k_1).
$$
(The $O(r^2)$ term corresponds to the area wasted due to the fact that some balls touch the border; note that, by construction, each strip has at most 3 balls that cover the area outside of the grid, namely the first one and possibly the two last ones.) Since the number of strips is $O(n/k_1)$, the total number of vertices wasted is 
$$
( O(k_2^2) + O(m^2 / k_1) ) \cdot O(n/k_1) = O(n^{3/2} \gamma^{1/2}) + O(n \gamma^{5/3}).
$$ 
It remains to show that the number of vertices that are \emph{not} wasted is at least $mn$. The number of vertices in balls with radii between $k_1$ and $k_2$ (in the perfect situation when one pretends they are all disjoint) is 
\begin{align*}
\sum_{r = k_1}^{k_2} & (1+2r+2r^2) = \frac{2(k_2+1)^3}{3}+\frac{k_2+1}{3} - \frac{2k_1^3}{3}-\frac{k_1}{3} > \frac 23 (k_2^3 - k_1^3) \\ 
&\ge  \frac 23 \left(  \frac 32 mn + \frac {9 C}{2 \gamma^{1/6}} mn - \frac {mn}{\gamma^{1/2}} \right) \ge mn + \frac {9C-2}{3 \gamma^{1/6}} mn = mn + \frac {9C-2}{3} n^{3/2} \gamma^{5/6},
\end{align*}
which is large enough to guarantee that even if $O(n^{3/2} \gamma^{1/2} + n \gamma^{5/3})$ vertices are ``wasted'', all vertices of the graph are covered: indeed, if $\gamma=O(1)$, then by choosing $C$ large enough, the term $\frac {9C-2}{3} n^{3/2} \gamma^{5/6}$ can be made bigger than the error term; for $\gamma=\omega(1)$ and $\gamma =o(n^{3/7})$, clearly $C=1$ is enough, since the dominating error term is $O(n^{3/2} \gamma^{1/2})$, which is clearly of smaller order than $n^{3/2} \gamma^{5/6}$, and finally for $\gamma=\Omega(n^{3/7})$ (but still $\gamma \leq \sqrt{n}$), also $C=1$ is enough, since the error term is  $O(n \gamma^{5/3})$, which is also of smaller order than $n^{3/2} \gamma^{5/6}$.

The case $m < \sqrt{n}$ is easy. Consider a path on $n$ vertices on the border of the grid. Set a fire on vertices at distance $\sqrt{n}$ from each other, and then after at most $2 \sqrt{n}$ steps the whole grid is burnt. The upper bound of $O(\sqrt{n})$ holds and the proof of Theorem~\ref{thm:mainGrid} is finished. $\hfill \square$

\bigskip

Let us notice that the proof of the lower bound works for the toroidal grid $C_m \square C_n$ with no adjustment needed. Moreover, as $P_m \square P_n$ is a spanning subgraph of $C_m \square C_n$, we have $b(C_m \square C_n) \le b(P_m \square P_n)$, and so the following corollary holds.

\begin{corollary}
$$b(C_m \square C_n) =
\begin{cases}
 (1+o(1)) (3/2)^{1/3} (mn)^{1/3}, & \text{ if } \sqrt{n} \ll m \le n, \\
 \Theta(\sqrt{n}), & \text{ if } m = O(\sqrt{n}).
\end{cases}
$$
\end{corollary}

\section{Proof of Theorem~\ref{thm:Path1}}\label{sec:drunkenness}

In this section we show our results on the cost of drunkenness for paths. 

\subsection{Proof of Theorem~\ref{thm:Path1}(i)}

We start with the proof of the upper bound for $b_1(P_n)$. Let $k=\sqrt{n(\log n / 2 + \omega)}$, where $\omega = \omega(n) = (\log n)^{2/3} = o(\log n)$ is a function tending to infinity as $n \to \infty$ sufficiently fast. We will use the first moment method to show that at time $k$, a.a.s.\ all vertices are burning. It will be convenient to think of covering the path $P_n$ with $k$ balls of radii $0, 1, \ldots, k-1$, where the radii are measured in terms of the graph distance, that is, the ball of radius $i$ around a vertex $v$ is $N(v,i)$. Partition $P_n$ into $\sqrt{n}$ subpaths, $p_1, p_2, \ldots, p_{\sqrt{n}}$, each of length $\sqrt{n}$. (For expressions such as $\sqrt{n}$ here that clearly have to be an integer, we round up or down but do not specify which: the choice of which does not affect the argument.) For $1 \le i \le \sqrt{n}$, let $X_i$ be the indicator random variable for the event that no ball contains the whole $p_i$. In other words, $X_i = 0$ if there exists $j \in [k]$ such that the ball centred at $x_j$ and of radius $j-1$ contains the path $p_i$; otherwise, $X_i = 1$. Let $X = \sum_{i=1}^{\sqrt{n}} X_i$. Clearly, if $X = 0$, then all vertices are burning at time $k$, and so our goal is to show that $\Prob(X \ge 1) = o(1)$.

First, we consider a subpath $p_i$ with $\sqrt{\log n} \leq i \leq \sqrt{n}-\sqrt{\log n}$. Note that all vertices of $p_i$ are at distance at least $(\sqrt{\log n}-1) \cdot \sqrt{n} > k$ from the endpoints of $P_n$. As a result, $p_i$ is sufficiently far from them to be affected by the boundary effect. It is clear that 
\begin{align*}
\E {X_i} & = \Prob (X_i = 1) = \prod_{j=\sqrt{n}/2}^k \left(1- \frac{2j-\sqrt{n}}{n}\right)\nonumber\\
& = \exp\left(- \left(1 + O\left(\frac{k}{n}\right)\right) \sum_{j=\sqrt{n}/2}^k \frac{2j-\sqrt{n}}{n} \right) \nonumber \\
& = \exp\left(- \left(1 + O\left(\frac{k}{n}\right)\right) \frac {(k-\sqrt{n}/2)^2}{n} \right) \nonumber \\
& = \exp\left(- \left(1 + O\left(\frac{\sqrt{n}}{k}\right)\right) \frac {k^2}{n} \right).
\end{align*}
Using the definition of $k$ and recalling that $\omega = (\log n)^{2/3}$, we get
\begin{align}
\E {X_i} = \Prob (X_i = 1) & = \exp\left(- \left(1 + O\left(\frac{1}{\sqrt{\log n}}\right)\right) \left( \frac 12 \log n + \omega \right) \right) \nonumber \\
& = \exp\left(- \frac 12 \log n - \omega + O(\sqrt{\log n}) \right) \nonumber \\
& \le \exp\left(- \frac 12 \log n - \frac 12 \omega \right). \label{eq6}
\end{align}

On the other hand for any $i<\sqrt{\log n}$ or $i> \sqrt{n}-\sqrt{\log n}$, $p_i$ is close to one of the endpoints of $P_n$, but one can nevertheless estimate the probability of $X_i=1$ as follows: 
\begin{align}
\E {X_i} = \Prob (X_i = 1) & \le \prod_{j=\sqrt{n}}^k \left(1- \frac{j}{n}\right)\nonumber\\
& = \exp\left(- \left(1 + O\left(\frac{k}{n}\right)\right) \sum_{j=\sqrt{n}}^k \frac{j}{n} \right) \nonumber \\
& = \exp\left(- \left(1 + O\left(\frac{k}{n}\right)\right) \frac {k^2\left( 1 + O(n/k^2)\right)}{2n} \right) \nonumber \\
& = \exp\left(- \left(1 + O\left(\frac{n}{k^2}\right)\right) \frac {k^2}{2n} \right) \nonumber \\
& \le \exp\left(- \frac 14 \log n - \frac 14 \omega \right). \label{eq7}
\end{align}

From~(\ref{eq6}) and~(\ref{eq7}) we conclude that,
\begin{eqnarray*}
\E {X} & \leq & \Big(\sqrt{n}-2\sqrt{\log n} \Big) \exp\left(- \frac 12 \log n - \frac 12 \omega \right)  + \ 2 \sqrt{\log n} \exp\left(- \frac 14 \log n - \frac 14 \omega \right) \\
&=& o(1).
\end{eqnarray*}
The upper bound holds a.a.s.\ by Markov's inequality.

\medskip

Now, we will show an asymptotically almost sure matching lower bound for $b_2(P_n)$. This will finish the proof, since $b_2 (P_n) \le b_1(P_n)$ (see~(\ref{eq:3b})). Let $k=\sqrt{n(\log n / 2 - \log \log n / 2 - \omega)}$, where now $\omega = \omega(n) = o(\log \log n)$ is any function tending to infinity as $n \to \infty$, arbitrarily slowly. This time, we will use the second moment method to show that at time $k$, a.a.s.\ at least one vertex is not burning. In fact, in order to avoid highly dependent events, we focus only on vertices that are at distance at least $2k$ from each other. 

For $1 \le i \le \ell := n / (2k) - 1 = \Theta(\sqrt{n/\log n})$, let $Y_i$ be the indicator random variable for the event that vertex $v_{2ki}$ is not burning at time $k$. Let $Y = \sum_{i=1}^{\ell} Y_i$. In particular, if $Y \ge 1$, then at least one vertex is not burning at time $k$, and so our goal is to show that $\Prob(Y = 0) = o(1)$. Recall that, since we investigate $b_2(P_n)$, at time $j$ of the process, only vertices not selected earlier have a chance to be selected as the next vertex $x_j$ to be burned. Recall also that the ball centred at $x_j$ will have radius $k-j$ at time $k$. Hence, for any $i \in [\ell]$, 
\begin{align*}
\Prob (Y_i = 1) & = \prod_{j=1}^k \left(1- \frac{2k-2j+1}{n-j+1} \right) \nonumber\\
& = \exp \left(- \left(1 + O \left( \frac{k}{n} \right) \right) \sum_{j=1}^k \frac{2k-2j+1}{n} \right)\nonumber\\
& = \exp\left(- \frac{k^2}{n}\right) \big(1 + o(1)\big) \nonumber \\
& = \exp\left(- \frac 12 \log n + \frac 12 \log \log n + \omega \right) \big(1 + o(1)\big),
\end{align*}
and so 
$$
\E{Y} = \ell \cdot \exp\left(- \frac 12 \log n + \frac 12 \log \log n + \omega \right) \big(1 + o(1)\big) = \Theta(e^{\omega}) \to \infty,
$$
as $n \to \infty$. Now, we estimate the variance of $Y$ as follows: 

\begin{align*}
\V [Y] & = \sum_{i,i'} \mathrm{Cov}{(Y_i, Y_i')} \nonumber\\
& \leq \sum_{i\neq i'}\left(\Prob(Y_i = Y_{i'}=1) - \big(\Prob(Y_i =1 )\big)^2\right) + \E{Y}.
\end{align*}
Let $i, i' \in [\ell]$ be such that $i \neq i'$. Since the vertices $v_{2ki}$ and $v_{2ki'}$ are far away from each other, by performing similar calculations as before, we get
\begin{align*}
\Prob (Y_i = Y_{i'}=1) & = \prod_{j=1}^k \left(1- 2 \cdot \frac{2k-2j+1}{n-j+1} \right) \nonumber\\
& = \exp\left(- \log n + \log \log n + 2 \omega \right) \big(1 + o(1)\big) \\
& = \big(\Prob(Y_i)\big)^2 \big(1 + o(1)\big).
\end{align*}
Therefore,
\begin{align*}
\V [Y] & \le O(\ell^2) \cdot o \Big( \exp\left(- \log n + \log \log n + 2 \omega \right) \Big) + \Theta(e^{\omega}) \\
& = o \Big( e^{2 \omega} \Big) = o \Big( \big(\E{Y} \big)^2\Big).
\end{align*}
The lower bound holds a.a.s.\ by Chebyshev's inequality.  $\hfill \square$

\bigskip

\subsection{Proof of Theorem~\ref{thm:Path1}(ii)}
Since $b_3(P_n) \ge b(P_n) = \Omega(\sqrt{n})$, we only need to show the matching upper bound. In order to warm up, we will show that a.a.s.
$$
b_3(P_n) = O( \sqrt{n} \log^* n),
$$ 
where $\log^* n$ denotes the \emph{iterated logarithm} of $n$, that is, the number of times the logarithm must be iteratively applied before the result is less than or equal to 1. After that, a simple trick will be enough to replace $\log^* n$ by a constant. 

Partition $P_n$ into $\sqrt{n}$ subpaths, $p_1, p_2, \ldots, p_{\sqrt{n}}$, each of length $\sqrt{n}$. We say that a given subpath is burning at some point of the process if some vertex of that subpath is burning. For $1 \le i \le \sqrt{n}$ and $t \in \N$, let $X_t(i)$ be the indicator random variable for the event that $p_i$ is \emph{not} burning at time $3 \sqrt{n} t - \sqrt{n}$. Let $X_t = \sum_{i=1}^{\sqrt{n}} X_t(i)$ be the random variable counting the number of subpaths \emph{not} burning at time $3 \sqrt{n} t - \sqrt{n}$. Finally, for $t \in \N \cup \{0 \}$, let $Y_t$ be the number of vertices \emph{not} burning at time $3\sqrt{n} t$. We say that we are in \emph{phase} $t$, if the number of time steps elapsed is in the set $\{3\sqrt{n}(t-1)+1, \ldots, 3 \sqrt{n}t\}$. Furthermore, in phase $t$, we say that we are in the first sub-phase, if the number of time steps elapsed is in the set $\{3\sqrt{n}(t-1)+1, \ldots, 2 \sqrt{n}t\}$, and in the second sub-phase, otherwise. Clearly, $Y_0=n$ (deterministically), since no vertex is burning at the beginning of the process. Note that $X_t$ is only defined for positive integers but, for convenience, we set $X_0 = \sqrt{n}$. Note also that for every $t \in \N \cup \{0\}$ we have 
$$
Y_t \le X_t \sqrt{n},
$$
since if at least one vertex of a subpath is burning, then after additional $\sqrt{n}$ steps the whole subpath is burning. We run the burning process and observe the sequence $X_1, Y_1, X_2, Y_2, \ldots$ ($X_1$ is determined at time $2\sqrt{n}$, $Y_1$ at time $3\sqrt{n}$, $X_2$ at time $5 \sqrt{n}$, and so on). Our goal is to get an upper bound for $X_{t+1}$ knowing $X_t$, which implies an upper bound for $Y_t$, as already mentioned above.

Note that the original random variables $X_{t+1}(i)$ are not independent. However, we are able to couple the original process with the following independent one. For each phase of length $3\sqrt{n}$ of the original  process, in the new process we do the following: in the first sub-phase of the phase, a vertex is chosen uniformly at random for burning, independently of the fact whether it is burnt or not, and no spreading takes place; during the last $\sqrt{n}$ steps of the phase, no new vertex is chosen for burning, only spreading takes place. Consider then the following coupling: if in the independent process (during the first sub-phase) a vertex not yet burnt in the original process is chosen for burning, the same vertex is also chosen for burning in the original process; otherwise, choose uniformly at random an unburnt vertex in the original model and burn it there. Spreading occurs deterministically in both processes (in case of the independent model it occurs only in the second sub-phase). Let us focus on the independent model for a moment. Note that, as in the original model, if a subpath $p_i$ is burning at the end of the first sub-phase, then it is completely burnt at the end of the second sub-phase (the end of a given phase). On the other hand, subpaths that are not burning at the end of the first sub-phase might or might not be burning at the end of the second sub-phase, depending whether they were adjacent to some burning path or not. For simplicity, at the end of the second sub-phase, only sub-paths that were burning at the end of the first sub-phase are (completely) burnt; that is, we change the status of all other vertices to unburnt so that no other subpath is burning. Sub-paths that are completely burnt by the end of the second sub-phase, are removed from the set of vertices in the independent model and not considered in following phases. Clearly, by the coupling, if a vertex is burnt in the independent model it is also burnt in the original model. This applies to every step of the process.

For $1 \le i \le \sqrt{n}$ and $t \in \N$, let $X'_t(i)$ be the indicator random variable for the event that $p_i$ is \emph{not} burning at time $3 \sqrt{n} t - \sqrt{n}$ in the independent process, and let $X'_t = \sum_{i=1}^{\sqrt{n}} X'_t(i)$ be the random variable counting the number of subpaths \emph{not} burning at time $3 \sqrt{n} t - \sqrt{n}$.  By the coupling, clearly $X_t(i) \leq X'_t(i)$ and $X_t \leq X'_t$. The distribution of $X'_t$, knowing $X'_{t-1}$, is easy to analyze. Suppose that $X'_{t-1}(i)=1$ for some $i$ (that is, $p_i$ is not burning at the beginning of phase $t$). Then,
\begin{eqnarray*}
p = p(X'_{t-1}) &:=& \Prob (X'_t(i) = 1) = \left( 1 - \frac {1}{X'_{t-1}} \right)^{2\sqrt{n}} \\
&=& \exp \left( - \frac {2\sqrt{n}}{X'_{t-1}} + O \left( \frac {\sqrt{n}} { (X'_{t-1})^2} \right) \right) = (1+o(1)) \exp \left( - \frac {2\sqrt{n}}{X'_{t-1}} \right),
\end{eqnarray*}
provided that $X'_{t-1} \gg n^{1/4}$. As events are independent, it follows that $X'_t$ is simply the binomial random variable $\text{Bin}(X'_{t-1}, p(X'_{t-1}))$.

We will show that the number of subpaths that are not burning decreases quickly, that is, we will show that for a relatively small value of $t$ we have $X'_t < 7 \sqrt{n} / \log n$. Suppose that for a given $t$ we have that $X'_t \ge 7 \sqrt{n} / \log n$. We get that 
$$
\E{X'_{t+1} ~~|~~ X'_t} = (1+o(1)) X'_t \exp \left( - \frac {2 \sqrt{n}}{X'_t} \right).
$$
We will now show that the probability that $X'_{t+1}$ is at least, say, $X'_t \exp \left( -\sqrt{n}/X'_t \right)$ is small (conditioning on the value of $X'_t$ and the fact that $X'_t \ge 7 \sqrt{n} / \log n$). In order to show it, we will use the following  version of Chernoff's bound (see e.g.~\cite[Theorem~2.1]{JLR}): if $X = \sum_{j=1}^{\ell} X_j$ is a sum of independent indicator random variables, each $X_j$ following a Bernoulli distribution with a (possibly) different probability of success, then for $\eps > 0$ we have 
\begin{equation*}
\Prob \Big( X-\E X \ge \eps \E X \Big) \le \exp \left( - \frac {\eps^2 \E X}{2+\eps} \right).
\end{equation*}
In particular, if $\eps \ge 1$, then
\begin{equation}\label{eq:chern3}
\Prob \Big( X-\E X \ge \eps \E X \Big) \le \exp \left( - \frac {\eps \E X}{3} \right).
\end{equation}
Noting that $X'_t \le \sqrt{n}$, we apply~(\ref{eq:chern3}) with 
$$
\eps = \frac { X'_t \exp(-\sqrt{n}/X'_t) - \E{X'_{t+1} ~~|~~ X'_t} }{ \E{X'_{t+1} ~~|~~ X'_t} } \ge (1+o(1)) \exp(\sqrt{n}/X'_t) - 1 \ge e +o(1) -1 \ge 1,
$$ 
to get that 
\begin{eqnarray*}
\Prob \Big( X'_{t+1} \ge X'_t \exp(-\sqrt{n}/X'_t) ~~|~~ X'_t \Big) &\le& \exp \Big( -\frac14 \ X'_t \exp(-\sqrt{n}/X'_t) \Big) \\
&\le& \exp \Big( -\frac 14 \cdot \frac {7 \sqrt{n}}{\log n} \exp( - \log n / 7) \Big) \\
&\le& \exp \Big( - n^{1/4} \Big),
\end{eqnarray*}
since $X'_t$ is assumed to be at least $7 \sqrt{n} / \log n$. 

Using this observation, our goal is to get (recursively) upper bounds for the sequence of random variables $X'_1, X'_2, \ldots$ as follows: 
\begin{eqnarray*}
X'_1 &\le& X'_0 / e = \sqrt{n} / e,\\
X'_2 &\le& X'_1 / \exp(\sqrt{n}/X'_1) \le \sqrt{n} / e^e,\\
\ldots &\le& \ldots \\
X'_{t+1} &\le& X'_{t} / \exp(\sqrt{n}/X'_t) \le \sqrt{n} / e\uparrow\uparrow t,
\end{eqnarray*}
provided that $X'_t \ge 7 \sqrt{n} / \log n$ ($e\uparrow\uparrow t$ denotes $t$-times iterated exponentiation, using Knuth's up-arrow notation). In other words, we condition on the fact that $X'_i \ge 7 \sqrt{n} / \log n$ (otherwise, we simply stop applying the argument) and $X'_i \le \sqrt{n} /  e\uparrow\uparrow i$, and estimate the probability that the desired bound for $X'_{i+1}$ fails. Since we apply the claim at most $\log^* (\log n / 7) \le \log^* n$ times and each time the claim fails with probability at most  $\exp ( - n^{1/4} ) = o(1/\log^* n)$, we get that a.a.s.\ after $T \le \log^* n$ rounds $X'_T \le 7 \sqrt{n} / \log n$. 

The rest of the proof is straightforward. Back in the original model, all but at most $Y_T \le X_T \sqrt{n} \le X'_T \sqrt{n} \le 7 n / \log n$ vertices are burning. Now, observe that it is enough to wait another $7 \sqrt{n}$ steps to see all subpaths burning a.a.s.: indeed, the probability that a given subpath is not burning at that point is at most
$$
\left( 1 - \frac {\sqrt{n}}{Y_T} \right)^{7 \sqrt{n}} \le \exp \left( - \frac {7n}{Y_T} \right) \le \frac {1}{n}.
$$ 
The expected number of subpaths that survived these additional $7\sqrt{n}$ steps is at most $X_T / n = o(1)$, and so the claim holds by Markov's inequality. After $\sqrt{n}$ more steps all vertices are burning (deterministically) and so a.a.s.\ the process ends in at most
$$
3\sqrt{n} \cdot T + 8\sqrt{n} = O(\sqrt{n} \log^* n)
$$
rounds in total.

\bigskip

Now, we are ready to modify the argument slightly to get an upper bound of $O(\sqrt{n})$ instead of $O(\sqrt{n} \log^* n)$. Before we used to have phases of equal length, namely, $3 \sqrt{n}$ steps, and we needed $O(\log^* n)$ phases to make sure that after $O(\sqrt{n} \log^* n)$ steps the number of vertices burning is small enough for the final argument to be applied. We can make the time intervals a bit shorter (from phase to phase): let now the length of phase $t$ to be only $4\sqrt{n} 2^{-t}$ steps, so that the total number of steps is still $O(\sqrt{n})$. We adjust the independent model as follows. Let $\ell=\ell(n)$ be the smallest integer such that $2^{\ell} \ge \sqrt{n}$; clearly, $\sqrt{n} \le 2^{\ell} < 2\sqrt{n}$. We start with subpaths of length $2^{\ell}$ instead of $\sqrt{n}$. The first (second) sub-phase of phase $t$ consists of $3 \sqrt{n} 2^{-t}$ ($\sqrt{n} 2^{-t}$, respectively) steps, instead of $2\sqrt{n}$ ($\sqrt{n}$, respectively). The only other difference is that at the end of each phase, we split each non-burning subpath into two; that is, at the end of phase $t$, each subpath has length $2^{\ell-t}$. As before, let $X'_t$ be the number of subpaths not burning at the end of the first sub-phase of phase $t$, now of length $2^{\ell-t}$, and note also that the number of unburnt vertices in the original model at the end of a phase is bounded from above by the number of unburnt vertices in the independent model. This time, we get $X'_0 = n/2^{\ell} \le \sqrt{n}$, and for each $t \in \N$, $X'_t \sim 2 \text{Bin}(X'_{t-1}, p(X'_{t-1}))$, where
\begin{eqnarray*}
p = p(X'_{t-1}) &:=& \left( 1 - \frac {1}{X'_{t-1}} \right)^{3 \sqrt{n} 2^{-t}} = (1+o(1)) \exp \left( - \frac {3 \sqrt{n}}{ 2^{t} X'_{t-1}} \right),
\end{eqnarray*}
provided that $2^t X'_{t-1} \gg n^{1/4}$.

Suppose that for a given $t$ we have $2^{\ell-t} X'_t \ge 7n / \log n$ (note that $2^{\ell-t} X'_t$ is the total number of unburnt vertices in the independent model) and $X'_t \le \sqrt{n} 2^{-t}$. We get that 
$$
\E {X'_{t+1} ~~|~~ X'_t} = (2+o(1)) X'_t \exp \left( - \frac {3 \sqrt{n}}{ 2^{t+1} X'_t} \right)
$$
and, applying Chernoff's bound~\eqref{eq:chern3} with 
\begin{eqnarray*}
\eps &=& \frac { X'_t \exp((-3\sqrt{n})/(2^{t+2}X'_t)) - \E{X'_{t+1}~~|~~ X'_t} }{ \E{X'_{t+1}~~|~~ X'_t}} \geq \left( \frac 12 + o(1) \right) \exp \left( \frac {3\sqrt{n}}{2^{t+2}X'_t} \right) -1 \\
&\geq& (1/2+o(1)) e^{3/4} - 1 \geq 1,
\end{eqnarray*}
we get
\begin{eqnarray*}
\Prob \Big( X'_{t+1} \ge X'_t \exp((-3\sqrt{n})/(2^{t+2}X'_t)) ~~|~~ X'_t \Big) &\le& \exp \Big( - \frac14 \ X'_t \exp((-3\sqrt{n})/(2^{t+2}X'_t)) \Big) \\
& \leq &  \exp \Big( - \frac14 \ \frac{7 n}{2^{\ell-t} \log n} \exp(- \frac {3 \sqrt{n}}{4 \cdot 2^{t}} \cdot \frac{2^{\ell-t} \log n}{7n}) \Big) \\
& \leq &  \exp \Big( - \frac{\sqrt{n}}{\log n} \exp \Big(- \frac{3}{14} 2^{-2t}\log n\Big) \Big) \\
&\le& \exp \Big( - n^{1/4} \Big).
\end{eqnarray*}
Thus, with probability at least $1-e^{-n^{1/4}}$, 
\begin{equation}\label{eq:bound}
X'_{t+1} < X'_t \exp((-3\sqrt{n})/(2^{t+2}X'_t)) \le \frac {X'_t}{2},
\end{equation}
provided that $7n 2^{t-\ell}/\log n \le X'_t \le \sqrt{n} 2^{-t}$, and which then also implies $X_{t+1} \le \sqrt{n} 2^{-(t+1)}$.  Note that the probability that the first inequality in~\eqref{eq:bound} does not hold is clearly $o(1/\log \log n)$, and there are  $O(\log \log n)$ phases (but still only $O(\sqrt{n})$ steps!), until the number of vertices burning is $O(n/\log n)$. Hence, a.a.s. after $T'=O(\log \log n)$ rounds we have $X'_{T'} \le 7\sqrt{n}/\log n$. The rest of the argument as before: we have, back in the original model, we have $Y_{T'} \leq X_{T'}2^{\ell-T'} \leq X'_{T'}2^{\ell-T'} \leq 7n/\log n$, and as before, a.a.s.\ we finish after $8\sqrt{n}$ additional steps. The proof is finished.$\hfill \square$


\end{document}